\documentclass[final]{siamltex}

\usepackage[dvips]{epsfig}

\newcommand{\bff}{{\bf f}}

\newcommand{\bxi}{{\mbox{\boldmath $\xi$}}}
\newcommand{\bzet}{{\mbox{\boldmath $\zeta$}}}

\newcommand{\bep}{{\mbox{\boldmath $\epsilon$}}}
\newcommand{\buzet}{{\underline{\mbox{\boldmath $\zeta$}}}}

\newcommand{\bud}{{\underline {\bf d}}}

\newcommand{\buf}{{\underline {\bf f}}}

\newcommand{\bd}{{\bf d}}
\newcommand{\bv}{{\bf v}}
\newcommand{\bw}{{\bf w}}
\newcommand{\bx}{{\bf x}}
\newcommand{\bU}{{\bf U}}

\newcommand{\bQ}{{\bf Q}}

\newcommand{\bS}{{\bf S}}
\newcommand{\bT}{{\bf T}}

\newcommand{\bD}{{\bf D}}

\newcommand{\bV}{{\bf V}}
\newcommand{\bZ}{{\bf Z}}

\newcommand{\R}{I \! \! R}
\newcommand{\N}{I \! \! {N}}

\newcommand{\ute}{{\underline \theta}}

\newcommand{\us}{{\underline s}}

\newcommand{\unu}{{\underline \nu}}
\newcommand{\ux}{{\underline x}}

\newcommand{\uf}{{\underline f}}
\newcommand{\ud}{{\underline d}}
\newcommand{\ue}{{\underline e}}

\newcommand{\uxi}{{\underline{\xi}}}
\newcommand{\uzet}{{\underline{\zeta}}}

\newcommand{\bb}{\begin{eqnarray}}
\newcommand{\be}{\end{eqnarray}}

\title{A diffusion equation for the density of the ratio of two jointly distributed Gaussian variables and the numerical inversion of Laplace transform}

\author{Piero Barone\thanks{Istituto per le Applicazioni del Calcolo ''M. Picone'',C.N.R.Via dei Taurini 19, 00185 Rome, Italy,({\tt p.barone@iac.cnr.it, piero.barone@gmail.com})
}}

\begin{document}

\maketitle

\begin{abstract}
It is shown that the density of the ratio of two  random variables with the same  variance and joint Gaussian density satisfies a  non stationary diffusion equation.
Implications of this result for kernel density estimation of the condensed density of the generalized eigenvalues of a  random matrix pencil useful for the numerical inversion of the Laplace transform is discussed.
\end{abstract}

\begin{keywords}
parabolic equations, random matrices, kernel estimation
\end{keywords}

\begin{AMS}
62G07, 35K05, 65R10
\end{AMS}

\pagestyle{myheadings}
\thispagestyle{plain}

\section*{Introduction}
The density of the ratio of two   random variables with joint bivariate Gaussian density has been derived by several authors and it is important in  many applications (see e.g. \cite{hink,mars1,mars2,pham}). In the sequel it is proved that, when the two variables have the same variance, this density satisfies a parabolic partial differential equation whose  coefficients depend on both the independent variables. The proof is based on standard properties of the confluent hypergeometric functions of the first kind. A motivation for deriving such a PDE is provided by the problem of the numerical inversion of Laplace transform from noisy discrete data \cite{bor,brs}. This is a classical ill-posed problem. Insights for its stable solution can be obtained from knowledge of the marginal densities of the damping factors of a multiexponential model which represents a discretization of the Laplace transform. This problem can be restated in terms of the condensed density of the generalized eigenvalues of a matrix pencil built from the observations. In a recent paper \cite{botev} an adaptive kernel density estimator  based on linear diffusion processes has been proposed which has several advantages over the existing methods. In the sequel  a kernel density estimator in the class considered in \cite{botev},  based on the proposed diffusion equation, for estimating the condensed density mentioned above is proposed.   A Montecarlo simulation allows to appreciate its merits with respect to a Gaussian kernel estimator and its effectiveness for the  numerical inversion of the Laplace transform.

The paper is organized as follows. In the first section the density of the ratio of two   random variables with joint bivariate Gaussian density is shortly derived in terms of confluent hypergeometric functions of the first kind. In the second section the PDE is derived. In the third section the kernel density estimator based on the PDE is derived and the conditions which need to be met by the function whose Laplace transform has to be inverted in order to get good results  are specified. In the last section the merits of the proposed method are  shown by a MonteCarlo simulation.

\section{The density of the ratio of two jointly distributed Gaussian variables}

Let us assume that the random variables $(\bv,\bw)$ have a joint Gaussian density
$$g(v,w)=\frac{1}{2\pi |\Sigma|^{\frac{1}{2}}} e^{-\frac{1}{2}[v-\nu_v, w-\nu_w]\Sigma^{-1}[v-\nu_v, w-\nu_w]^T}$$ with $|\Sigma|>0$ and the mean, the covariance matrix and its inverse are given by:
\begin{eqnarray} \unu=[\nu_v, \nu_w]^T,\;\;\;\;\;\Sigma=
\left[\begin{array}{llll}
 \sigma^2_v & \;\gamma\\
 \gamma & \sigma^2_w\
  \end{array}\right],\;\;\;\;\;\Sigma^{-1}=\frac{1}{\sigma^2_v\sigma^2_w-\gamma^2}
\left[\begin{array}{llll}
 \sigma^2_w & -\gamma\\
 -\gamma & \;\sigma^2_v\
  \end{array}\right].\nonumber
\end{eqnarray}
Let   $_1F_1[\alpha,\beta,z]$  be the confluent hypergeometric function of the first kind. The following lemma holds:
\begin{lemma}
If $a\in\R^+, \;b\in\R,\;\forall n\in\N$
\begin{eqnarray} L_n=\int_{-\infty}^\infty|\lambda|\lambda^n e^{-a\lambda ^2+2b \lambda }d\lambda\nonumber=
\left\{\begin{array}{llll}
 a^{-\frac{2+n}{2}}\Gamma[\frac{2+n}{2}] {_1}F_1\left[\frac{2+n}{2},\frac{1}{2},\frac{b^2}{
  a}\right], \mbox{  $n$  even}\\
 2ba^{-\frac{3+n}{2}}  \Gamma[\frac{3+n}{2}] {_1}F_1\left[\frac{3+n}{2},
   \frac{3}{2}, \frac{b^2}{a}\right], \mbox{ $n$ odd}
  \end{array}\right.
\end{eqnarray}
\label{lem1}
\end{lemma}

{\em Proof}.
Let us define  $g(\lambda)=e^{-a\lambda ^2+2b \lambda }$, then if $n$ is even
\begin{eqnarray} \int_{-\infty}^\infty|\lambda|\lambda^n g(\lambda)d\lambda=
\int_{-\infty}^0|\lambda^{n+1}| g(\lambda)d\lambda+\int_0^{\infty}\lambda^{n+1} g(\lambda)d\lambda= \nonumber\\
 \int_0^{\infty}\lambda^{n+1} g(-\lambda)d\lambda+\int_0^{\infty}\lambda^{n+1} g(\lambda)d\lambda=\int_0^{\infty}\lambda^{n+1} [g(\lambda)+ g(-\lambda)]d\lambda\nonumber
\end{eqnarray}
if  $n$ is odd
\begin{eqnarray} \int_{-\infty}^\infty|\lambda|\lambda^n g(\lambda)d\lambda=
-\int_{-\infty}^0|\lambda^{n+1}| g(\lambda)d\lambda+\int_0^{\infty}\lambda^{n+1} g(\lambda)d\lambda= \nonumber\\
- \int_0^{\infty}\lambda^{n+1} g(-\lambda)d\lambda+\int_0^{\infty}\lambda^{n+1} g(\lambda)d\lambda=
 \int_0^{\infty}\lambda^{n+1}[ g(\lambda)-g(-\lambda)]d\lambda\nonumber.
\end{eqnarray}
But (see e.g. \cite[3.462,1]{gr})
$$\int_0^{\infty}\lambda^{n+1}g(\lambda)d\lambda=(2 a)^{-\frac{n+2}{2}}\Gamma(n+2)e^{\frac{4 b^2}{8a}}D_{-(n+2)}\left(\frac{-\sqrt{2} b}{\sqrt{ a}}\right)$$
where the parabolic cylinder function $D_{-(n+2)}(z)$ is given by
$$D_{-(n+2)}(z)=2^{-\frac{n+2}{2}}e^{-\frac{z^2}{4}}\left(\frac{\sqrt{\pi}}{\Gamma\left(\frac{n+3}{2}\right)} {_1}F_1\left[\frac{n+2}{2},\frac{1}{2},\frac{z^2}{2}\right] -
\frac{\sqrt{2\pi}z}{\Gamma\left(\frac{n+2}{2}\right)} {_1}F_1\left[\frac{n+3}{2},\frac{3}{2},\frac{z^2}{2}\right]
 \right)$$
hence we get
$$\int_0^{\infty}\lambda^{n+1}g(\lambda)d\lambda=$$
$$2^{-(n+2)}\sqrt{\pi}a^{-\frac{n+2}{2}}\Gamma(n+2)
\left(\frac{1}{\Gamma\left(\frac{n+3}{2}\right)} {_1}F_1\left[\frac{n+2}{2},\frac{1}{2},\frac{ b^2}{ a}\right] +
\frac{2b}{\sqrt{a}\Gamma\left(\frac{n+2}{2}\right)} {_1}F_1\left[\frac{n+3}{2},\frac{3}{2},\frac{ b^2}{ a}\right]
 \right)$$
 and
 $$\int_0^{\infty}\lambda^{n+1} [g(\lambda)+ g(-\lambda)]d\lambda=$$
 $$2^{-(n+2)}\sqrt{\pi}a^{-\frac{n+2}{2}}\Gamma(n+2)
\left(\frac{2}{\Gamma\left(\frac{n+3}{2}\right)} {_1}F_1\left[\frac{n+2}{2},\frac{1}{2},\frac{ b^2}{ a}\right]\right)$$
 $$\int_0^{\infty}\lambda^{n+1} [g(\lambda)- g(-\lambda)]d\lambda=$$
 $$2^{-(n+2)}\sqrt{\pi}a^{-\frac{n+3}{2}}\Gamma(n+2)
 \left(\frac{4b}{\Gamma\left(\frac{n+2}{2}\right)} {_1}F_1\left[\frac{n+3}{2},\frac{3}{2},\frac{ b^2}{ a}\right]
 \right)$$
 but
 $$\frac{\sqrt{\pi}\Gamma(n+2)}{2^{(n+1)}\Gamma\left(\frac{n+3}{2}\right)}=\Gamma\left(\frac{n+2}{2}\right)$$
 $$\frac{\sqrt{\pi}\Gamma(n+2)}{2^{(n+1)}\Gamma\left(\frac{n+2}{2}\right)}=\Gamma\left(\frac{n+3}{2}\right).\endproof$$
 We can get the density of the ratio of Gaussian variables as a simple consequence of this Lemma (see also \cite{pham}):
\begin{theorem}
If $\Sigma>0$, the density of the ratio $\bx=\frac{\bw}{\bv}$ is  given by:
\begin{eqnarray}h(x)=\frac{e^{-c}}{2\pi |\Sigma|^{\frac{1}{2}}a(x)} {_1}F_1\left[1, \frac{1}{2}, \frac{b(x)^2}{a(x)}\right]
\end{eqnarray}
 where
\begin{eqnarray}
a(x)&=&\frac{1}{2|\Sigma|}\left(\sigma^2_w -2 \gamma x+\sigma^2_v x^2\right)\nonumber\\
b(x)&=&\frac{1}{2|\Sigma|}( \sigma^2_w \nu_v- \gamma \nu_w -\gamma \nu_v x + \sigma^2_v \nu_w x)\nonumber\\
c&=&\frac{1}{2|\Sigma|}\left(\sigma^2_w \nu_v^2-2 \gamma\nu_w \nu_v+\sigma^2_v \nu_w^2\right)\nonumber \\
|\Sigma|&=&\sigma^2_v\sigma^2_w-\gamma^2\nonumber
\end{eqnarray}
and $a(x)>0 \;\forall x, \;\;c>0$.
\end{theorem}
\begin{proof}
The density of the ratio $\bx=\frac{\bw}{\bv}$ can be written as:
$$h(x)=\frac{1}{2\pi |\Sigma|^{\frac{1}{2}}} \int_{-\infty}^\infty\int_{-\infty}^\infty \delta\left(x-\frac{w}{v}\right) e^{-\frac{1}{2}[v-\nu_v, w-\nu_w]\Sigma^{-1}[v-\nu_v, w-\nu_w]^T} dvdw.$$
By the change of variables
$\lambda=v,\;\;\mu=\frac{w}{v}$ with Jacobian  $|\lambda|$ we get
\begin{eqnarray}h(x)&=&\frac{1}{2\pi |\Sigma|^{\frac{1}{2}}}\int_{-\infty}^\infty\int_{-\infty}^\infty |\lambda| \delta\left(x-\mu\right) e^{-\frac{1}{2}[\lambda-\nu_v, \lambda\mu-\nu_w]\Sigma^{-1}[\lambda-\nu_v, \lambda\mu-\nu_w]^T}
d\lambda d\mu\nonumber \\ &=&\frac{1}{2\pi |\Sigma|^{\frac{1}{2}}} \int_{-\infty}^\infty|\lambda|e^{-\frac{1}{2}[\lambda-\nu_v, \lambda x-\nu_w]\Sigma^{-1}[\lambda-\nu_v, \lambda x-\nu_w]^T}d\lambda\nonumber \\ &=&
\frac{1}{2\pi |\Sigma|^{\frac{1}{2}}} \int_{-\infty}^\infty|\lambda|e^{-a(x)\lambda ^2+2b(x) \lambda -c}d\lambda\nonumber
\end{eqnarray}
Moreover $a(x)>0 \;\forall x$ because the quadratic equation in $x$ $$\sigma^2_w -2 \gamma x+\sigma^2_v x^2=0$$ has no real roots as $\Sigma>0$, hence, by Lemma \ref{lem1}, with $n=0$
\begin{eqnarray}h(x)&=&\frac{e^{-c}}{2\pi |\Sigma|^{\frac{1}{2}}} \int_{-\infty}^\infty|\lambda|e^{-a(x)\lambda ^2+2b(x) \lambda }d\lambda\nonumber=\frac{e^{-c}}{2\pi |\Sigma|^{\frac{1}{2}}a(x)} {_1}F_1\left[1, \frac{1}{2}, \frac{b(x)^2}{a(x)}\right].
\end{eqnarray}
Finally we notice that $c>0$ as $\sigma^2_w \nu_v^2-2 \gamma\nu_w \nu_v+\sigma^2_v \nu_w^2\ge (\sigma_w \nu_v-\sigma_v \nu_w)^2>0$, because $\Sigma>0.$ \end{proof}
\begin{corollary}
If $\nu_v=\nu_w=0$ and $\Sigma=\sigma^2 I$ then $h(x)=\frac{1}{\pi(x^2+1)}$.
\end{corollary}
\begin{proof}
In the considered case we have
\begin{eqnarray}
a(x)=\frac{1+x^2}{2\sigma^2},\;\;b(x)=0,\;\;c=0,\;\;|\Sigma|=\sigma^4,
\nonumber
\end{eqnarray}
hence
\begin{eqnarray}h(x)=\frac{1}{2\pi \sigma^2a(x)} {_1}F_1\left[1, \frac{1}{2}, 0\right]=\frac{1}{\pi (1+x^2)}.
\end{eqnarray}
\end{proof}

\section{The diffusion equation for the density of the ratio of two jointly distributed Gaussian variables}

Let us assume that
\begin{eqnarray} \Sigma=
\left[\begin{array}{llll}
 \sigma^2 & \;\gamma\\
 \gamma & \;\sigma^2\
  \end{array}\right]=\sigma^2\left[\begin{array}{llll}
 1 & \;\rho\\
 \rho &\; 1\
  \end{array}\right],\;\;\;|\rho|< 1\nonumber
\end{eqnarray}
and define $t=\sigma^2$. By making explicit the dependence on $t$ in  $a(x),b(x),c,|\Sigma|,h(x)$ we get
\begin{eqnarray}
a(x,t)=\frac{1-2\rho x +x^2}{2(1-\rho^2)t}\nonumber\\
b(x,t)=\frac{ \nu_v- \rho \nu_w  +  (\nu_w -\rho \nu_v )x}{2(1-\rho^2)t}\label{abc}\\
c(t)=\frac{ \nu_v^2-2 \rho\nu_w \nu_v+ \nu_w^2}{2(1-\rho^2)t} \nonumber\\
d(t)=|\Sigma|=(1-\rho^2)t^2\nonumber
\end{eqnarray}
and
\begin{eqnarray}h(x,t)=\frac{e^{-c(t)}}{2\pi d(t)^{\frac{1}{2}}a(x,t)} {_1}F_1\left[1, \frac{1}{2}, \frac{b(x,t)^2}{a(x,t)}\right]. \label{hform}
\end{eqnarray}
\noindent\underline{Remark}.
We notice that $$h(x,t;\nu_v,\nu_w,\rho)=h(x,\alpha^2 t;\alpha\nu_v,\alpha\nu_w,\rho), \;\forall \alpha\in\R.$$ Therefore if $\nu_v\neq 0$ and $\alpha=\frac{1}{\nu_v}$ we have \begin{eqnarray}h(x,t;\nu_v,\nu_w,\rho)=
h(x,\frac{t}{\nu_v^2};1,\frac{\nu_w}{\nu_v},\rho).\label{hform1}\end{eqnarray}
We have
\begin{theorem}
$$\lim_{t\rightarrow \infty}h(x,t)=\frac{\sqrt{1-\rho ^2}}{\pi  \left(x^2-2 x \rho
   +1\right)}=h(x,t;0,0,\rho)$$
   and, if $\nu_v\neq 0$,
$$\lim_{t\rightarrow 0^+}h(x,t)=\delta\left(x-\frac{\nu_w}{\nu_v}\right)
$$in the weak sense.
\label{teo22}
\end{theorem}
\begin{proof}
The density $h(x,t)$ can be rewritten as
\begin{eqnarray*}h(x,t)&=&\frac{e^{-\frac{(\nu_w-\nu_v x)^2}{2 t
   \left(x^2-2 x \rho +1\right)}}}{\sqrt{2 \pi t \left(x^2-2 x \rho
   +1\right)}}\frac{\nu_v (1-\rho x)  +\nu_w (x-\rho ) }{ x^2-2 x
   \rho +1 }
   \mbox{erf}\left[\frac{\nu_v (1-\rho x)
   +\nu_w (x-\rho )}{\sqrt{2 t \left(1-\rho ^2\right) \left(x^2-2 x \rho
   +1\right)}}\right] \\
   &+& \frac{\sqrt{1-\rho ^2}}{\pi  \left(x^2-2 x \rho
   +1\right)} e^{\frac{-\nu_v^2+2 \nu_v \nu_w \rho
   -\nu_w^2}{2 t \left(1-\rho ^2\right)}}. \nonumber
   \end{eqnarray*}
Taking the limit $t\rightarrow \infty$ in this expression we get the first equality in the first part of the thesis. The second equality is obtained by substituting $\nu_v=\nu_w=0$ in equation (\ref{hform}). To prove the second part we notice that
$$\lim_{t\rightarrow 0^+}h(x,t)=\lim_{t\rightarrow 0^+}\frac{e^{-\frac{(\nu_w-\nu_v x)^2}{2 t
   \left(x^2-2 x \rho +1\right)}}}{\sqrt{2 \pi t \left(x^2-2 x \rho
   +1\right)}}\frac{\nu_v (1-\rho x)  +\nu_w (x-\rho ) }{ x^2-2 x
   \rho +1 }=\lim_{t\rightarrow 0^+}h_0(x,t)$$
and
$$h_0(x,t)=\frac{1}{\sqrt{t}}\tilde{h}_0\left(\frac{x}{\sqrt{t}},t\right) \mbox{  where  $\tilde{h}_0(x,t)$ is such that } \int_{-\infty}^{\infty} \tilde{h}_0(x,t)dx=1.$$
But then
$$\lim_{t\rightarrow 0^+}\int_{-\infty}^\infty h_0(x,t)F\left(x-\frac{\nu_w}{\nu_v}\right)dx=F\left(\frac{\nu_w}{\nu_v}\right)$$
holds for all continuous compactly supported functions $F$, and so $h(x,t)$ converges weakly to $\delta\left(x-\frac{\nu_w}{\nu_v}\right)$ in the sense of measures
( \cite[Theorem 1.18]{stwe}).
\end{proof}

\noindent The properties of $h(x,t)$ stated above  suggest, when $\nu_v\neq 0$, the existence of a diffusion equation ruling the behavior of $h(x,t)$ for varying $t$ (when $\nu_v = \nu_w= 0,$ $h(x,t)$ does not depend on $t$). To prove that this is indeed the case we need
the following Lemmas:
\begin{lemma}
If $\nu_v\neq 0$ and $|\rho|<1$ then
$$h_t(x,t)=\frac{e^{-c(t)}}{2\pi }\left[ A^{(t)}(x,t)L_2(x,t) +B^{(t)}(x,t)L_1(x,t)+C^{(t)}(t)L_0(x,t)\right]$$
$$h_x(x,t)=\frac{e^{-c(t)}}{2\pi }\left[ A^{(x)}(x,t)L_2(x,t) +B^{(x)}(x,t)L_1(x,t) \right]$$
$$h_{xx}(x,t)=\frac{e^{-c(t)}}{2\pi }\left[ E^{(xx)}(x,t)L_4(x,t) +F^{(xx)}(x,t)L_3(x,t) +A^{(xx)}(x,t)L_2(x,t)\right]$$
where
$$A^{(t)}(x,t)=\frac{1 + x^2 - 2 x \rho}{2 t^3 (1 - \rho^2)^{3/2}}$$
$$B^{(t)}(x,t)=-\frac{\nu_v + \nu_w x - (\nu_w + \nu_v x) \rho}{t^3 (1 - \rho^2)^{3/2}}$$
$$C^{(t)}(t)= \frac{\nu_v^2 + \nu_w^2 - 2 \nu_v \nu_w \rho +
 2 t (-1 + \rho^2)}{2 t^3 (1 - \rho^2)^{3/2}}$$
$$A^{(x)}(x,t)=\frac{\rho-x}{t^2 (1 - \rho^2)^{3/2}}$$
$$B^{(x)}(x,t)=\frac{\nu_w - \nu_v \rho}{t^2 (1 - \rho^2)^{3/2}}$$
$$E^{(xx)}(x,t)=\frac{(x - \rho)^2}{t^3 (1 - \rho^2)^{5/2}}$$
$$F^{(xx)}(x,t)=\frac{2 (x - \rho) (-\nu_w + \nu_v \rho)}{t^3 (1 - \rho^2)^{5/2}}$$
$$A^{(xx)}(x,t)=\frac{(\nu_w - \nu_v \rho)^2 + t (\rho^2-1)}{t^3 (1 - \rho^2)^{5/2}}.$$
\label{lem2}
\end{lemma}
\begin{proof}
We have
\begin{eqnarray} \Sigma^{-1}=\frac{1}{(1-\rho^2)t}
\left[\begin{array}{llll}
 \;\;1 & -\rho\\
 -\rho & \;\;1\
  \end{array}\right]\nonumber
\end{eqnarray}
therefore
$$h_t(x,t)=\frac{1}{2\pi } \int_{-\infty}^\infty\int_{-\infty}^\infty \delta\left(x-\frac{w}{v}\right)\frac{\partial}{\partial t}\frac{ e^{-\frac{1}{2}[v-\nu_v, w-\nu_w]\Sigma^{-1}[v-\nu_v, w-\nu_w]^T}}{\sqrt{d(t)} }dvdw.$$
By the change of variables
$\lambda=v,\;\;\mu=\frac{w}{v}$ with Jacobian  $|\lambda|$ we get
$$h_t(x,t)=\frac{1}{2\pi } \int_{-\infty}^\infty\int_{-\infty}^\infty |\lambda|\delta\left(x-\mu\right)\frac{\partial}{\partial t}\frac{ e^{-\frac{1}{2}[\lambda-\nu_v, \lambda\mu-\nu_w]\Sigma^{-1}[\lambda-\nu_v, \lambda\mu-\nu_w]^T}}{\sqrt{d(t)} }d\lambda d\mu=$$
$$\frac{1}{2\pi } \int_{-\infty}^\infty |\lambda|\frac{\partial}{\partial t}\frac{ e^{-a(x,t)\lambda ^2+2b(x,t) \lambda -c(t)}}{\sqrt{d(t)} }d\lambda =$$
$$\frac{1}{2\pi }\left( A^{(t)}(x,t)\int_{-\infty}^\infty |\lambda|\lambda^2 f(\lambda)  d\lambda +B^{(t)}(x,t)\int_{-\infty}^\infty |\lambda|\lambda f(\lambda)  d\lambda +C^{(t)}(x,t)\int_{-\infty}^\infty |\lambda|f(\lambda)  d\lambda\right)$$
where $f(\lambda)=e^{-a(x,t)\lambda ^2+2b(x,t) \lambda -c(t)}=e^{-c(t)}g(\lambda)$, and $a(x,t),b(x,t),c(t)$ are given in equations (\ref{abc}),  and
$$A^{(t)}(x,t)=\frac{1 + x^2 - 2 x \rho}{2 t^3 (1 - \rho^2)^{3/2}}$$
$$B^{(t)}(x,t)=-\frac{\nu_v + \nu_w x - (\nu_w + \nu_v x) \rho}{t^3 (1 - \rho^2)^{3/2}}$$
$$C^{(t)}(t)= \frac{\nu_v^2 + \nu_w^2 - 2 \nu_v \nu_w \rho +
 2 t (-1 + \rho^2)}{2 t^3 (1 - \rho^2)^{3/2}}$$
In the same way we get
$$h_x(x,t)=\frac{1}{2\pi }\left( A^{(x)}(x,t)\int_{-\infty}^\infty |\lambda|\lambda^2 f(\lambda)  d\lambda +B^{(x)}(x,t)\int_{-\infty}^\infty |\lambda|\lambda f(\lambda)  d\lambda \right)$$
where
$$A^{(x)}(x,t)=\frac{\rho-x}{t^2 (1 - \rho^2)^{3/2}}$$
$$B^{(x)}(x,t)=\frac{\nu_w - \nu_v \rho}{t^2 (1 - \rho^2)^{3/2}}$$
and
$$h_{xx}(x,t)=\frac{1}{2\pi } E^{(xx)}(x,t)\int_{-\infty}^\infty |\lambda|\lambda^4 f(\lambda)  d\lambda +\frac{1}{2\pi }F^{(xx)}(x,t)\int_{-\infty}^\infty |\lambda|\lambda^3 f(\lambda)  d\lambda+$$ $$\frac{1}{2\pi }A^{(xx)}(x,t)\int_{-\infty}^\infty |\lambda|\lambda^2 f(\lambda)  d\lambda$$
where
$$E^{(xx)}(x,t)=\frac{(x - \rho)^2}{t^3 (1 - \rho^2)^{5/2}}$$
$$F^{(xx)}(x,t)=\frac{2 (x - \rho) (-\nu_w + \nu_v \rho)}{t^3 (1 - \rho^2)^{5/2}}$$
$$A^{(xx)}(x,t)=\frac{(\nu_w - \nu_v \rho)^2 + t (\rho^2-1)}{t^3 (1 - \rho^2)^{5/2}}.$$
By using the same notations of Lemma \ref{lem1} we get the thesis.
\end{proof}
\begin{lemma}
If $a\in\R^+, \;b\in\R$,
$$L_3(x,t)=W_1L_1(x,t)+W_2L_2(x,t)$$ and $$L_4(x,t)=W_3L_1(x,t)+W_4L_2(x,t)$$
where
$$W_1=\frac{3}{2a},\;\;
W_2=\frac{b}{a},\;\;
W_3=\frac{3b}{2a^2},\;\;
W_4=\frac{2a+b^2}{a^2}$$
\label{lem3}
\end{lemma}
\begin{proof}
By Lemma \ref{lem1} we have
$$L_1=2ba^{-2}   {_1}F_1\left[2,\frac{3}{2}, \frac{b^2}{a}\right]$$
$$L_2=a^{-2} {_1}F_1\left[2,\frac{1}{2},\frac{b^2}{a}\right]$$
$$L_3= 4ba^{-3}   {_1}F_1\left[3,\frac{3}{2}, \frac{b^2}{a}\right]$$
$$L_4=2a^{-3} {_1}F_1\left[3,\frac{1}{2},\frac{b^2}{a}\right]$$
By \cite[13.4.3]{as} we have:
$$ _1F_1[h+1,k,z]=\frac{h-k+1}{h}\; _1F_1[h,k,z]-\frac{1-k}{h}\; _1F_1[h,k-1,z]$$
but then
$$L_3= 4ba^{-3}{_1}F_1\left[3,
   \frac{3}{2}, \frac{b^2}{a}\right]=ba^{-3}
   \left(3 _1F_1\left[2,\frac{3}{2},\frac{b^2}{a}\right]+ {_1}F_1\left[2,\frac{1}{2},\frac{b^2}{a}\right]   \right)=$$
$$ba^{-3}
   \left(3\frac{a^2}{2b}L_1 + a^2L_2   \right).$$
Moreover, from by \cite[13.4.5]{as} we have:
$${_1}F_1\left[h+1,k,z\right]=\frac{(h+z)}{h} {_1}F_1\left[h,k,z\right]-\frac{(k-h)z}{hk}{_1}F_1\left[h,k+1,z\right]$$
and therefore
$$L_4=2a^{-3} {_1}F_1\left[3,\frac{1}{2},\frac{b^2}{
  a}\right]=
  2a^{-3}\left( \frac{(2+\frac{b^2}{
  a})}{2} {_1}F_1\left[2,\frac{1}{2},\frac{b^2}{
  a}\right]+\frac{3}{2}\frac{b^2}{
  a} {_1}F_1\left[2,\frac{3}{2},\frac{b^2}{
  a}\right]\right)=$$
  $$
  \frac{3b}{2a^2}L_1+\frac{2a+b^2}{a^2} L_2.$$
  \end{proof}
  
\noindent We can now prove the main theorem:
\begin{theorem}
If $\nu_v\neq 0$ and $|\rho|<1$, the density $h(x,t)$ solves the partial differential equation
\begin{eqnarray}h_t(x,t)&=&\mathcal D h(x,t) \\
\mathcal D&=&\frac{\partial [D(x,t)\frac{\partial \bullet}{\partial x}]}{\partial x}+C(x,t) \frac{\partial \bullet}{\partial x} +S(t) \bullet\label{oper}\end{eqnarray}
where the diffusion coefficient is
\begin{eqnarray}D(x,t)=\frac{P_3(x)}{Q_1(x)+tQ_2(x)} \label{coedif}\end{eqnarray}
 the source coefficient is
\begin{eqnarray*}S(t)=C^{(t)}(t) d(t)^\frac{1}{2}\end{eqnarray*}
and the convection coefficient is
\begin{eqnarray*}C(x,t)=\frac{P_1(x) + tP_2(x)}{t (Q_1(x) + tQ_2(x))} -
\frac{P_3'(x)}{Q_1(x) + tQ_2(x)} + \frac{P_3(x) (
Q_1'(x) +
Q_2'(x))}{(Q_1(x) + tQ_2(x))^2}
\end{eqnarray*}
where
\begin{eqnarray*}P_1(x)&=&2 (\nu_w - \nu_v x)^2 [\nu_v + \nu_w x - (\nu_w + \nu_v x) \rho] ( \rho^2-1)  \\
P_2(x) &=&   (1 + x^2 -
      2 x \rho) [\nu_w ( \rho-x) ( 3 x^2 - 6 x \rho +
         11 \rho^2-8 )+ \\
 &&\nu_v (2 - 9 x^2 + 10 x \rho + 3 x^3 \rho - 5 \rho^2 -
         x \rho^3)]\\
P_3(x)&=&(1 + x^2 - 2 x \rho)^2 \{\nu_w (1 - x^2 + 2 x \rho - 2 \rho^2) +
     \nu_v [\rho +
        x (-2 + x \rho)]\} \\
Q_1(x)&=&2  (1 - \rho^2) (\nu_w - \nu_v x)^2 (\nu_w -
        \nu_v \rho)\\
Q_2(x)&=& 2  ( \rho^2-1) \{\nu_w (1 + 4 x^2 - 8 x \rho + 3 \rho^2) -
        \nu_v [\rho + x (3 x^2 - 5 x \rho + \rho^2)]\}.
\end{eqnarray*}
$Q_1(x) + tQ_2(x)$ is a cubic polynomial  with one, two or three real zeros depending on the values of $t,\nu_v,\nu_w,\rho$.
\label{teo2}
\end{theorem}
\begin{proof}
Dropping the dependencies on $(x,t)$, by Lemma \ref{lem2} we have
$$h_{xx}=\frac{e^{-c(t)}}{2\pi }\left[ E^{(xx)}L_4 +F^{(xx)}L_3 +A^{(xx)}L_2\right]$$ and by Lemma
\ref{lem3} we have
$$h_{xx}=\frac{e^{-c(t)}}{2\pi }\left[ E^{(xx)}(W_3L_1+W_4L_2) +F^{(xx)}(W_1L_1+W_2L_2) +A^{(xx)}L_2\right]=$$
$$\frac{e^{-c(t)}}{2\pi }\left[(A^{(xx)}+F^{(xx)}W_2+ E^{(xx)}W_4)L_2+(F^{(xx)}W_1+E^{(xx)}W_3)L_1\right].$$
By Lemma \ref{lem2} we have
$$h_x=\frac{e^{-c(t)}}{2\pi }\left[ A^{(x)}L_2 +B^{(x)}L_1 \right]$$
we can then solve formally for $L_1,L_2$ the linear system
\begin{eqnarray*}\frac{e^{-c(t)}}{2\pi }\left[\begin{array}{llll}
 A^{(x)} & \;B^{(x)}\\
 C^{(xx)} & D^{(xx)}
  \end{array}\right]
  \left[\begin{array}{llll}
 L_2\\
 L_1
  \end{array}\right]=
 \left[\begin{array}{llll}
 h_x \\
 h_{xx}
  \end{array}\right]
\end{eqnarray*}
where $$C^{(xx)}=A^{(xx)}+F^{(xx)}W_2+ E^{(xx)}W_4,\;\;\;D^{(xx)}=F^{(xx)}W_1+E^{(xx)}W_3.$$
We get
\begin{eqnarray*}
 L_2=e^{c(t)}2\pi \frac{D^{(xx)}h_x-B^{(x)}h_{xx}}{A^{(x)}D^{(xx)}-B^{(x)}C^{(xx)}}\\
 L_1=e^{c(t)}2\pi \frac{-C^{(xx)}h_x+A^{(x)}h_{xx}}{A^{(x)}D^{(xx)}-B^{(x)}C^{(xx)}}
\end{eqnarray*}
Substituting these expression in
 $$h_t=\frac{e^{-c(t)}}{2\pi }\left[ A^{(t)}L_2 +B^{(t)}L_1+C^{(t)}L_0\right]$$
and remembering that
$$h=\frac{e^{-c(t)}}{2\pi d^\frac{1}{2}}L_0$$
we get
\begin{eqnarray}h_t =C^{(t)}(t) d(t)^\frac{1}{2}h+G_x h_x +G_{xx} h_{xx} \label{ht} \end{eqnarray}
where
$$G_x= \frac{A^{(t)} D^{(xx)}-B^{(t)} C^{(xx)}}{ A^{(x)} D^{(xx)}-B^{(x)} C^{(xx)}}$$
$$G_{xx}=\frac{A^{(x)} B^{(t)}-A^{(t)} B^{(x)}}{ A^{(x)} D^{(xx)}-B^{(x)} C^{(xx)}}.$$
Substituting the expressions for $ A^{(t)}, A^{(x)}, B^{(t)}, B^{(x)}, C^{(xx)}, D^{(xx)}, E^{(xx)}, F^{(xx)}$ given in Lemma \ref{lem2}
and noticing that
$$C(x,t)=G_x-\frac{\partial G_{xx}}{\partial x}\;\;\mbox{ and  }\;\;D(x,t)=G_{xx}$$
we get the expressions reported above.
Moreover  $Q_1(x)+tQ_2(x)=0$ is a cubic polynomial equation whose discriminant can be positive, negative or zero depending on the values of $t,\nu_v,\nu_w,\rho$.
\end{proof}

\section{A density estimation problem}

Let
$$L_f(s)=\int_0^\infty f(t)e^{-st}dt$$
be the Laplace transform of a function $f(t)\in L_1(\R^+)$. Let us denote random quantities by bold characters. Let be $$\bd_k=L_f(k \Delta_s)+\bep_k,\;\;\;\Delta_s>0,\;\;\;k=1,\dots,n$$ where $\bep_k$ are i.i.d. Gaussian zero mean random variables with variance $\sigma^2$ and let us consider the problem of making inference on $f(t)$ from $R$ independent realizations of $\bud=[\bd_1,\dots,\bd_n]$.
The problem can be severely ill-posed. An approach to its solution consists in approximating  the Laplace transform by a finite sum, assuming $n$ even
$$L_f(k \Delta_s)\approx \sum_{j=1}^pf_j e^{-\alpha_j (k-1)},\;\;\alpha_j>0,\;\;p=\frac{n}{2}$$ and in solving
for the unknowns $\{f_j,\alpha_j\},\;j=1,\dots,p$ in the multiexponential model (for simplicity the same symbols are used):
$$\bd_k=\sum_{j=1}^pf_j e^{-\alpha_j (k-1)}+\bep_k=\sum_{j=1}^pf_j \zeta_j^{k-1}+\bep_k.$$
 In the noiseless case the problem consists in
interpolating the data
\begin{eqnarray}s_k=\sum_{j=1}^pf_j e^{-\alpha_j (k-1)},\;\;k=1,\dots,n\label{sign}\end{eqnarray} by
 means of a linear combination of real exponential
 functions $\zeta_j(t)=e^{-\alpha_j t},
 \;\;j=1,\dots,p$.
To this aim let us consider the Hankel matrices
$$U_0(\us)=U(s_0,\dots,s_{n-2}),\;\;\;\;
  U_1(\us)=U(s_1,\dots,s_{n-1})$$ where
$$ U(x_1,\dots,x_{n-1})=\left[\begin{array}{llll}
x_1 & x_{2} &\dots &x_{p} \\
x_{2} & x_{3} &\dots &x_{p+1} \\
. & . &\dots &. \\
x_{p} & x_{p+1} &\dots &x_{n-1}
  \end{array}\right]$$
It is well known (e.g.\cite{hen2}) that, provided that $\det(U_0)\ne 0,
\det(U_1)\ne 0$, a unique solution exists. If $\uxi$ and $W$ denote the generalized eigenvalues and  eigenvectors of the matrices  $(U_1,U_0)$ then the solution is given by
$$\uzet=\uxi,\;\;\uf=W^T\us=V(\uxi)^{-1}\us$$ where
$V(\uxi)$ is the square Vandermonde matrix
based on $\uxi$ and $T$ denotes
transposition. Hence the critical quantities which the solution depend on are the generalized eigenvalues $\uxi$. They can be computed by the generalized Schur decomposition  of the matrices  $(U_1,U_0)$ \cite{horn}:
$$ U_1 = QSZ^T,\;\;\;U_0 = QTZ^T $$
where $Q$ and $Z$ are orthogonal matrices, and $S$ and $T$ are  upper triangular matrices such that $\xi_j=\frac{S_{jj}}{T_{jj}}$.
In the noisy case the matrices $\bU_0,\bU_1$ are random and the generalized eigenvalues $\bxi_j,\;j=1,\dots,p$ are random variables. Their marginal densities are all equal to the their condensed density (see e.g. \cite[Lemma 2.4]{bor}) which is defined as
\begin{eqnarray}H(x)=E\left[\frac{1}{p}\sum_{j=1}^p\delta(x-\bxi_j)\right].\label{condens}\end{eqnarray}
Knowledge of the condensed density is therefore of main importance for making inference on the generalized eigenvalues $\uxi$.

In a more general context this problem was studied in \cite{j08} where a stochastic perturbation method for estimating the condensed density (\ref{condens}) based on a single realization of $\bud$ was proposed. Here we assume to have $R$ independent realizations $\ud^{(r)},\;\;r=1,\dots,R$ of $\bud$ and we are seeking  a kernel estimator of the marginal densities.
In a recent paper \cite{botev} it has been shown that kernel estimators based on  parabolic partial differential equations can be considered and the underlying PDE can be used to estimate the optimal bandwidth and to take into account some kinds of prior information through suitable boundary conditions. Gaussian kernels belong to this class as they satisfy the heat equation. In the specific case considered here the Gaussian kernel estimator of (\ref{condens}) takes the form
\begin{eqnarray}\hat{H}_G(x,t)=\frac{1}{R}\sum_{r=1}^R \frac{1}{p_r}\sum_{k=1}^{p_r}\Phi(x,\xi_k^{(r)},t)\label{Gest}\end{eqnarray}
where
$$\Phi(x,\mu, t) = \frac{1}{\sqrt{2 \pi t}}
e^{-(x-\mu )^2/(2 t)}$$
where $\xi_k^{(r)},k=1,\dots,p_r$ are the real generalized eigenvalues of $(U_1^{(r)},U_0^{(r)})$ built from $\ud^{(r)}$ (discarding the complex conjugate pairs).
It turns out that $\hat{H}_G(x,t)$ is the unique solution of the diffusion equation $$\frac{\partial }{\partial t}\hat{H}_G(x,t)=\frac{1}{2}\frac{\partial^2 }{\partial x^2}\hat{H}_G(x,t)$$
with initial condition $\hat{H}_G(x,0)=H_e(x)$ where
$$H_e(x)=\frac{1}{R}\sum_{r=1}^R\frac{1}{p_r}\sum_{k=1}^{p_r}\delta(x-\xi_k^{(r)})$$
is the empirical condensed density of the generalized eigenvalues.

We now notice that if $p=1$ the only generalized eigenvalue $\bxi=\bd_2/\bd_1$ is the ratio of two uncorrelated Gaussian random variables with the same variance $t=\sigma^2$ and mean $f_1 \zeta_1$ and $f_1$ respectively and its density was derived in Section 1. Moreover in Section 2 a diffusion equation was derived which is satisfied by this density. The idea is then to replace the standard diffusion operator which gives rise to a Gaussian kernel density estimation with a more specific diffusion operator related to the one defined in Theorem  \ref{teo2}. However we can not use straightforwardly the operator (\ref{oper}) because   the theory developed in \cite{botev} holds for diffusion operators with  coefficients independent of $t$ and positive diffusion coefficient.
On the other hand when $p>1$ the generalized eigenvalues are  the ratio of variables which are not Gaussian. Therefore in any case, when proposing a modified operator based on  (\ref{oper}), we are looking for a suboptimal solution to the kernel selection problem. However it turns out that  the generalized eigenvalues can be approximated by the ratio of Gaussian variables and the approximation errors of the numerator and denominator are random variables whose expectation and standard deviation  are proportional to $$\mathcal E=\frac{\sigma^2}{\prod_{i=1}^p f_i \prod_{i<j}(\zeta_i-\zeta_j)^6}.$$
  This will be proved in Theorem \ref{approx}. Hence the approximation can be very good if the signal-to-noise ratio, measured by $\frac{\prod_{i=1}^p |f_i|}{\sigma^2}$, is large enough with respect to the relative distance of the numbers $\zeta_i,\;i=1,\dots,p$,  measured by $\prod_{i<j}(\zeta_i-\zeta_j)^6$.

A modified operator can be built as follows. We first notice that the difficulty of the Laplace inversion problem strongly depends on the relative position of the $\zeta_j,\;j=1,\dots,p$ which the interpolation of the noiseless data is based on. Simplistically the closer they are the worse the conditioning of the problem is. We then prove that in these difficult cases
the diffusion coefficient of the operator (\ref{oper}) is positive in a neighbor of the interesting region of the density for $\sigma$ small enough. This is proved in Theorem \ref{pos}. We first need the following
\begin{lemma}
The  generalized eigenvalues of the random pencil $(\bU_1,\bU_0)$ built from the data $\bud=[\bd_1,\dots,\bd_n]$ are given by
 $\bD_0=\mbox{ diag}(\buf)$, $\bD_1=\mbox{ diag}(\buf)\cdot\mbox{ diag}(\buzet)$ where $\bzet_j,\bff_j,j=1,\dots,p$ are random variables such that $d_k(\omega)=\sum_{j=1}^p f_j(\omega) \zeta_j^{(k-1)}(\omega),\;k=1,\dots,n,\;\forall \omega\in\Omega$ where $\Omega$ is the space of events. If the generalized Schur decomposition of $(\bU_1,\bU_0)$ is given by $$ \bU_1 = \bQ\bS\bZ^T,\;\;\;\bU_0 = \bQ\bT\bZ^T $$ then
$$\mbox{ diag}(\bS)=\bD_1,\;\;\mbox{ diag}(\bT)=\bD_0.$$
\label{lem4}
\end{lemma}
\begin{proof}
Let $\bff_j, \bzet_j,\;j=1,\dots,p$ be the solution of the exponential interpolation problem which exists and it is unique a.s. because $\det(\bU_0)\ne 0$ a.s and $\det(\bU_1)\ne 0$ a.s. \cite{cinese}. If $\bV$ is the Vandermonde matrix $\bV_{ij}=\bzet_i^j$ then (see e.g. \cite{j08}) $$\bU_0=\bV\bD_0\bV^T,\;\;\bU_1=\bV\bD_1\bV^T.$$ But then
$$\bU_1\bV^{-T}\bD_0=\bU_0\bV^{-T}\bD_1.$$ Therefore the pairs $(\bff_j\bzet_j,\bff_j),j=1,\dots,p$ are representatives of the projective form \cite{stw} of the generalized eigenvalues of $(\bU_1,\bU_0)$ and the thesis follows.
\end{proof}
\begin{theorem}
 If $\rho_j=\mbox{ corr}(\bS_{jj},\bT_{jj})$   then for $h\neq k$
  $$\lim_{|\zeta_h-\zeta_k|\rightarrow 0}\lim_{\sigma\rightarrow 0}\rho_j=1,\;\forall j$$  and it exists an open interval $I\subset\R^+$   such that
$\frac{\nu_w}{\nu_v}\in I$ and $D(x,t)>0, \;x\in I,\; \forall t$.
\label{pos}
\end{theorem}
\begin{proof}
By Lemma \ref{lem4}   we have
\begin{eqnarray}\rho_j&=&E\left[\frac{\bS_{jj}-E[\bS_{jj}]}{\sqrt{E[(\bS_{jj}-E[\bS_{jj}])^2]}}\cdot \frac{\bT_{jj}-E[\bT_{jj}]}{\sqrt{E[(\bT_{jj}-E[\bT_{jj}])^2]}}\right]\nonumber\\ &=& E\left[\frac{\bff_j\bzet_j-E[\bff_j\bzet_j]}{\sqrt{E[(\bff_j\bzet_j-E[\bff_j\bzet_j])^2]}}\cdot \frac{\bff_j-E[\bff_j]}{\sqrt{E[(\bff_j-E[\bff_j])^2]}}\right].\label{cor}\end{eqnarray}
For each realization, $\bzet_j$ and $\bff_j$ are analytic functions of $\bd_k$ in a small neighbor of $\us$  (\cite[Lemma 2]{j08}), therefore they admit  Taylor series expansions around $\us$
$$\bzet_j=\zeta_j+\sum_{i=1}^n g_i \bep_i+\frac{1}{2}\sum_{i,h=1}^n C_{ih}\bep_i\bep_h+\dots$$
$$\bff_j=f_j+\sum_{i=1}^n c_i \bep_i+\frac{1}{2}\sum_{i,h=1}^n G_{jih}\bep_i\bep_h+\dots.$$
Truncating after the first order terms and substituting these expressions in (\ref{cor}), after some long but simple calculations, we get
\begin{eqnarray*}\lim_{\sigma\rightarrow 0}\rho_j=\frac{\zeta_j\sum_{i=1}^n c_i^2 +f_j\sum_{i=1}^n c_i g_i }{\left\{\left(\zeta_j\sum_{i=1}^n c_i^2+ f_j\sum_{i=1}^n c_i g_i \right)^2+f_j^2\left(\sum_{i=1}^n c_i^2\sum_{i=1}^n g_i^2 -\left(\sum_{i=1}^n c_i g_i \right)^2\right)\right\}^{1/2}} .\end{eqnarray*}
But if for some $h\neq k,\;\;|\zeta_h-\zeta_k|\rightarrow 0$ then $|c_i|\rightarrow \infty\;\forall i.$ In fact
$$g_i=\frac{\partial \bzet_j}{\partial \bd_i}_{|\bd=\us},\;\;c_i=\frac{\partial \bff_j}{\partial \bd_i}_{|\bd=\us}$$
and
$$c_i=\frac{\partial \left(\ue_j^T V(\uzet)^{-T}\us\right)}{\partial s_i}=\frac{\partial \left(\sum_{k=1}^p v_{jk}(\uzet)s_k\right)}{\partial s_i}=
\sum_{k=1}^p\frac{\partial v_{jk}(\uzet)}{\partial s_i}s_k+v_{ji}(\uzet)$$
where $v_{jk}=\ue_j^TV^{-T}\ue_k$. But (see e.g. \cite{planc})
$$\frac{\partial v_{jk}(\uzet)}{\partial s_i}=-\ue_j^T\left(V^{-T}\frac{\partial V^T(\uzet)}{\partial s_i}V^{-T} \right)\ue_k=
-\sum_{h=1}^p g_h\ue_j^T\left(V^{-T}\frac{\partial V^T(\uzet)}{\partial \zeta_h}   V^{-T} \right)\ue_k$$ and therefore
$$c_i=-\sum_{h,k=1}^p g_h s_k\ue_j^T\left(V^{-T}\frac{\partial V^T(\uzet)}{\partial \zeta_h}   V^{-T} \right)\ue_k +v_{ji}(\uzet).$$
The first part of the thesis then follows by noticing that $$\det{V}=\prod_{h\neq k}^{1,p}(\zeta_h-\zeta_k)$$ and
$$\lim_{(|c_1|,\dots,|c_n|)\rightarrow\infty}\frac{\zeta_j\sum_{i=1}^n c_i^2 +f_j\sum_{i=1}^n c_i g_i }{\left\{\left(\zeta_j\sum_{i=1}^n c_i^2+ f_j\sum_{i=1}^n c_i g_i \right)^2+f_j^2\left(\sum_{i=1}^n c_i^2\sum_{i=1}^n g_i^2 -\left(\sum_{i=1}^n c_i g_i \right)^2\right)\right\}^{1/2}}=1$$
because $\zeta_j>0$.
To prove the second part, let us consider the Taylor first order approximation of the diffusion coefficient around $\rho=1$ and $x=\frac{\nu_w}{\nu_v}$:
\begin{eqnarray*}D(x,t)&=&A(t)
   +\left(x-\frac{\nu_w}{\nu_v}\right) \left(B(t)
   +O(1-\rho)^2\right) +O(1-\rho)^2
   +O\left(x-\frac{\nu_w}{\nu_v}\right)^2
   \end{eqnarray*}
where
$$A(t)=\frac{(\nu_v-\nu_w)^4}{4 (1-\rho) \left(t \nu_v^4\right)}+\frac{(\nu_v-\nu_w)^2 \left(\nu_v^2+6 \nu_v \nu_w+\nu_w^2\right)}{8 t
   \nu_v^4}+\frac{(1-\rho) (\nu_v+\nu_w)^4}{16 \left(t \nu_v^4\right)}$$ and
\begin{eqnarray*}B(t)&=&\frac{7 (\nu_v-\nu_w)^3}{4 t \nu_v^3 (\rho -1)}-\frac{(\nu_v-\nu_w)
   \left(\nu_v^2+28 \nu_w \nu_v+7 \nu_w^2\right)}{8 \left(t \nu_v^3\right)}+\\&&\frac{\left(9 \nu_v^4+2 \nu_w \nu_v^3+14 \nu_w^3 \nu_v+7 \nu_w^4\right)
   (\rho -1)}{16 t \nu_v^4-16 t \nu_v^3 \nu_w} .\end{eqnarray*}
   But $A(t)>0$ as $\nu_v,\nu_w$ have the same sign because $\zeta_j>0$. Therefore we get the thesis  by the permanence of sign theorem.
\end{proof}

By using Theorem \ref{pos}  we can define the modified operator as  the operator (\ref{oper}) where  the coefficients are evaluated at a  fixed suitable  value $t_0$.
When $p=1$ the variable $t$ represents the common variance of the numerator and denominator of the generalized eigenvalue. In order to choose $t_0$ we can then look for the element in the set of densities (\ref{hform1}) which best fits the empirical condensed density $H_e(x)$, i.e.
$$(t_0,\theta_0)=\mbox{argmin}_{t,\theta}\|h(x,t;\theta)-H_e(x)\|^2_2$$
where $\theta=\{\frac{\nu_v}{\nu_w},\rho\}$.
Let us denote by $$\mathcal D_0= \frac{\partial [D(x,t_0)\frac{\partial \bullet}{\partial x}]}{\partial x}+C(x,t_0) \frac{\partial \bullet}{\partial x} +S(t_0) \bullet$$ this modified operator and define the kernel estimator
\begin{eqnarray}\hat{H}_P(x,t^*)=\frac{1}{R}\sum_{r=1}^R \frac{1}{p_r}\sum_{k=1}^{p_r}h^{(r,k)}(x,t^*)\label{Pest}\end{eqnarray} where
\begin{itemize}
\item
$h^{(r,k)}(x,t^*)$ is obtained by equation (\ref{hform1}) by replacing $\frac{\nu_w}{\nu_v}$ by $\frac{S_{kk}^{(r)}}{T_{kk}^{(r)}}$ obtained by computing the  generalized eigenvalues by the Schur decomposition of the matrices $(U_1^{(r)},U_0^{(r)})$ built from $\ud^{(r)}$,  taking the $p_r$ real ones (discarding the complex conjugate pairs), and by replacing  $\rho$ with the sample correlation coefficient $\hat\rho$ of the pooled real $S_{kk}^{(r)},T_{kk}^{(r)},\;r=1,\dots,R;\;k=1,\dots,p_r;$
\item
the optimal bandwidth is  given by \cite[eq.23]{botev}
$$t^*=\left(\frac{E\left[\left(\sqrt{D(x,t_0)}\right)^{-1}\right]}{2 R \sqrt{\pi}\|\mathcal D_0 h(x,t_0)\|^2_2}\right)^{2/5};$$
where
$E\left[\left(\sqrt{D(x,t_0)}\right)^{-1}\right]$
is estimated  by
$$\frac{1}{R}\sum_{i=1}^R \frac{1}{p_r}\sum_{k=1}^{p_r}\left(\sqrt{D^{(r,k)}(x,t_0)}\right)^{-1},$$
 $\|\mathcal D_0 h(x,t_0)\|^2_2$ is estimated  by
$$\frac{1}{R}\sum_{i=1}^R\frac{1}{p_r}\sum_{k=1}^{p_r} \|h_t^{(r,k)}(x,t_0)\|_2^2$$
and $\|h_t^{(r,k)}(x,t_0)\|_2^2$
is computed by numerical quadrature;
\item  $ D^{(r,k)}(x,t_0)$ denotes  the diffusion coefficient computed by
replacing in formula (\ref{coedif}) $\nu_w,\nu_v,\rho$ by the the same values used for $h^{(r,k)}(x,t^*)$. With the same substitutions we obtain $h_t^{(r,k)}(x,t_0)$  by formula (\ref{ht});
\end{itemize}
By the second part of Theorem \ref{teo22}, $\hat{H}_P(x,t)$ is the unique solution of the diffusion equation
\begin{eqnarray}\frac{\partial}{\partial t}\hat{H}_P(x,t)&=&\mathcal D_0 \hat{H}_P(x,t) \label{operm}\end{eqnarray}
with initial condition $\hat{H}_P(x,0)=H_e(x)$.

In the next Theorem   conditions under which the distribution of the generalized eigenvalues is well approximated by the distribution of the ratio of Gaussian variables are specified.
\begin{theorem}
 The generalized eigenvalues $(\bff_j\bzet_j,\bff_j),j=1,\dots,p$ of $(\bU_1,\bU_0)$
 are given by
 $$\bff_j\bzet_j=f_j\zeta_j+\sum_{i=1}^n h_{ji} \bep_i+\eta^{(1)}_j(\ux)$$
$$\bff_j=f_j+\sum_{i=1}^n c_{ji} \bep_i+\eta^{(2)}_j(\ux)$$
where $h_{ji}$ and $c_{ji}$ do not depend on $\uf$,  $\ux$ is a point of $\R^n$ lying in the interior of the line segment joining $\ud$ and $\us$, and
$$E[\eta^{(h)}_j(\us)]\le \frac{\sigma^2}{2}\frac{F_1(\uzet)}{\prod_{r=1}^p f_r\prod_{r\ne s}(\zeta_r-\zeta_s)^6},\;h=1,2$$
$$var[\eta^{(h)}_j(\us)]\le \frac{\sigma^4}{2}\frac{F_2(\uzet)}{\prod_{r=1}^p f_r^2\prod_{r\ne s}(\zeta_r-\zeta_s)^{12}},\;h=1,2$$
 where $F_h(\cdot),\;h=1,2$ are polynomials in $\zeta_1,\dots,\zeta_p$ .
\label{approx}
\end{theorem}

\begin{proof}
Let  $\Phi:\R^n\rightarrow\R^n$  be the map that associates to each $n-$vector the $n/2$ pairs corresponding to the projective form of the generalized eigenvalues of the pencil $(U_1,U_0)$ built from the $n-$vector. It was proved in \cite[Lemma 2]{j08} that $\Phi$ is analytic.
We can then consider the first order Taylor series expansions with remainder of $\bzet_j$ and $\bff_j$, as functions of $\bud$, around $\us$  (\cite[Th. B]{serf}):
$$\bzet_j=\zeta_j+\sum_{i=1}^n g_{ji} \bep_i+\frac{1}{2}\sum_{i,h=1}^n G_{jih}\bep_i\bep_h$$
$$\bff_j=f_j+\sum_{i=1}^n c_{ji} \bep_i+\frac{1}{2}\sum_{i,h=1}^n C_{jih}\bep_i\bep_h$$
where
$$g_{ji}=\frac{\partial \bzet_j}{\partial \bd_i}_{|\bd=\us},\;\;G_{jih}=\frac{\partial^2 \bzet_j}{\partial \bd_i\partial \bd_h}_{|\bd=\ux}$$
$$c_{ji}=\frac{\partial \bff_j}{\partial \bd_i}_{|\bd=\us},\;\;C_{jih}=\frac{\partial^2 \bff_j}{\partial \bd_i\partial \bd_h}_{|\bd=\ux}.$$
We notice that $g_{ji}=\frac{\partial \zeta_j}{\partial s_i}$ and analogously for $c_{ji},G_{jih},C_{jih}$.
Let us denote by $f_{ji}^{(h)}=\frac{\partial^{h} f_j}{\partial s_i^{h}}$ and $\zeta_{ji}^{(h)}=\frac{\partial^{h} \zeta_j}{\partial s_i^{h}}$.
Let be
$$V=\left[\begin{array}{llll}
1 & 1& \dots & 1 \\
\zeta_1 & \zeta_{2} &\dots &\zeta_{p} \\
\zeta_1^2 & \zeta_2^2 &\dots &\zeta_p^2 \\
. & . &\dots &. \\
\zeta_1^{n-1} & \zeta_2^{n-1} &\dots &\zeta_p^{n-1}
  \end{array}\right]$$
  and
$$V_i^{(1)}=\frac{\partial V}{\partial s_i}=\left[\begin{array}{llll}
0 & 0& \dots & 0 \\
\zeta_{1i}^{(1)} & \zeta_{2i}^{(1)} &\dots &\zeta_{pi}^{(1)} \\
2\zeta_1\zeta_{1i}^{(1)} & 2\zeta_2\zeta_{2i}^{(1)} &\dots &2\zeta_p\zeta_{pi}^{(1)} \\
. & . &\dots &. \\
(n-1)\zeta_1^{n-2}\zeta_{1i}^{(1)} & (n-1)\zeta_2^{n-2}\zeta_{2i}^{(1)} &\dots &(n-1)\zeta_p^{n-2}\zeta_{pi}^{(1)}
  \end{array}\right].$$
By derivating both members of  equation (\ref{sign}) with respect to $s_i$ we have
\begin{eqnarray}\frac{\partial \us}{\partial s_i}=\ue_i=\frac{\partial V \uf}{\partial s_i}=V_i^{(1)}\uf+V \uf_{i}^{(1)}=\tilde{V}D^{(1)}_{\zeta i}\uf+V \uf_{i}^{(1)}\label{der}\end{eqnarray}
where $D^{(1)}_{\zeta i}$ is the diagonal matrix built from $\zeta_{1i}^{(1)},\dots,\zeta_{pi}^{(1)}$ and
$$\tilde{V}=\left[\begin{array}{llll}
0 & 0& \dots & 0 \\
1 & 1 &\dots &1 \\
2\zeta_1 & 2\zeta_2 &\dots &2\zeta_p \\
. & . &\dots &. \\
(n-1)\zeta_1^{n-2} & (n-1)\zeta_2^{n-2} &\dots &(n-1)\zeta_p^{n-2}
  \end{array}\right].$$
 But then if $$\ute=[\uf^T,\uzet^T]^T$$ and $D_f$ is the diagonal matrix built from $f_1,\dots,f_p$, equation (\ref{der}) becomes
\begin{eqnarray}[V \;\vdots\;\tilde{V}D_f]\ute^{(1)}_i=W \ute^{(1)}_i=\ue_i.\label{der1}\end{eqnarray}
and therefore $$c_{ji}=\ue_j^T[I\;\vdots\;0]W^{-1}\ue_i,\;\;\;g_{ji}=\ue_j^T[0\;\vdots\;I]W^{-1}\ue_i.$$ We then have
$$h_{ji}=\frac{\partial(f_j\zeta_j)}{\partial s_i}=\ue_j^T(\zeta_j[I\;\vdots\;0]+f_j [0\;\vdots\;I])W^{-1}\ue_i.$$
But
$$W=[V \;\vdots\;\tilde{V}]\left[\begin{array}{llll}I & 0\\0 & D_f\end{array}\right],\;\;\;W^{-1}=\left[\begin{array}{llll}I & 0\\0 & D_f^{-1}\end{array}\right][V \;\vdots\;\tilde{V}]^{-1}$$
hence
$$c_{ji}=\ue_j^T[I\;\vdots\;0]\left[\begin{array}{llll}I & 0\\0 & D_f^{-1}\end{array}\right][V \;\vdots\;\tilde{V}]^{-1}\ue_i=
\ue_j^T[I\;\vdots\;0][V \;\vdots\;\tilde{V}]^{-1}\ue_i$$
 is a function of $\uzet$ only (it does not depend on $\uf$).
As  $$h_{ji}=([\zeta_j\ue_j^T\;\vdots\;0]+[0\;\vdots\;\ue_j^T])[V \;\vdots\;\tilde{V}]^{-1}\ue_i$$
 does not depend on $\uf$ we get the first part of the thesis.

\noindent Let be
$$G_{jih}=\frac{\partial^2 \bzet_j}{\partial \bd_i\partial \bd_h}_{|\bd=\us}$$
$$C_{jih}=\frac{\partial^2 \bff_j}{\partial \bd_i\partial \bd_h}_{|\bd=\us}$$
where, for simplicity, the same symbols as before were used, and let be
$$\eta^{(1)}_j(\us)=\frac{1}{2}\sum_{i,h=1}^n H_{jih}\bep_i\bep_h$$
$$\eta^{(2)}_j(\us)=\frac{1}{2}\sum_{i,h=1}^n C_{jih}\bep_i\bep_h$$
where
$$H_{jih}=C_{jih}\zeta_j+2 c_{ji} g_{ji}+f_j G_{jih}$$
because
$$\frac{\partial \bff_j\bzet_j}{\partial \bd_i}_{|\bd=\us}=c_{ji}\zeta_j+f_j g_{ji}.$$
But then
$$E[\eta^{(1)}_j(\us)]=\frac{\sigma^2 }{2}\mbox{ tr}(H_j),\;\;\;E[\eta^{(2)}_j|\ux]=\frac{\sigma^2 }{2} \mbox{ tr}(C_j)$$
and, by Isserlis's theorem,
$$Var[\eta^{(1)}_j(\us)]=\frac{\sigma^4}{2}\sum_{i,j}^{1,n}H_{jih}^2$$
$$Var[\eta^{(2)}_j(\us)]=\frac{\sigma^4}{2}\sum_{i,j}^{1,n}C_{jih}^2$$
To conclude the proof we need an expression for $C_{jih}$ and $G_{jih}$. If
$$\Gamma_{ih}=\frac{\partial^2 \ute}{\partial s_i\partial s_h}=[C_{*ih}^T \vdots G_{*ih}^T]^T$$
by e.g. \cite[Ch.5]{planc} we have
$$C_{jih}=\ue_j^T[I\;\vdots\;0] \Gamma_{ih}=\ue_j^T[I\;\vdots\;0]\frac{\partial}{\partial s_h}[W^{-1}\ue_i]=-\ue_j^T[I\;\vdots\;0] W^{-1}\frac{\partial W}{\partial s_h}W^{-1}\ue_i$$
and $$G_{jih}=-\ue_j^T[0\;\vdots\;I] W^{-1}\frac{\partial W}{\partial s_h}W^{-1}\ue_i$$
where
$$W^{(1)}=\frac{\partial W}{\partial s_h}=\left[\tilde{V}D^{(1)}_{\zeta i}\vdots \tilde{V}D^{(1)}_{f i} +  \check{V}D_fD^{(1)}_{\zeta i} \right]$$
and $$\check{V}=\left[\begin{array}{llll}
0 & 0& \dots & 0 \\
0 & 0 &\dots &0 \\
2 & 2 &\dots &2 \\
. & . &\dots &. \\
(n-1)(n-2)\zeta_1^{n-3} & (n-1)(n-2)\zeta_2^{n-3} &\dots &(n-1)(n-2)\zeta_p^{n-3}
  \end{array}\right]$$
  and $D^{(1)}_{\zeta i}$ and $D^{(1)}_{f i}$ are the diagonal matrices built respectively from $\zeta_{1i}^{(1)},\dots,\zeta_{pi}^{(1)}$ and $f_{1i}^{(1)},\dots,f_{pi}^{(1)}$.
  We now notice that the elements of $[V \;\vdots\;\tilde{V}]^{-1}$
are rational functions of $\zeta_1,\dots,\zeta_p$. More specifically by \cite{shou}
$$[V \;\vdots\;\tilde{V}]^{-1}=D_W^{-1}\cdot X,\;\;\;\;
 D_W=\left[\begin{array}{ll}
D^3 &  0 \\
0 & D^2 \end{array}\right]$$
where
$$D=\mbox{ diag}\left[\prod_{i\ne 1}(\zeta_i-\zeta_1),\dots,\prod_{i\ne p}(\zeta_i-\zeta_p)\right]$$  and the elements of $X$ are polynomials in $\zeta_1,\dots,\zeta_p$.
But \begin{eqnarray*}\mbox{ diag}\left[\ute^{(1)}_i\right]=\mbox{ diag}\left[W^{-1}\ue_i\right]=\left[\begin{array}{ll}
D^{(1)}_{f i} &  0 \\
0 & D^{(1)}_{\zeta i} \end{array}\right]= \\ \left[\begin{array}{llll}I & 0\\0 & D_f^{-1}\end{array}\right]\left[\begin{array}{ll}
D^{-3} &  0 \\
0 & D^{-2} \end{array}\right]
\mbox{ diag}\left[X\ue_i\right]= \\ \left[\begin{array}{llll}D^{-3}D_{Xfi} & 0\\0 & D_f^{-1}D^{-2}D_{X\zeta i}\end{array}\right]\end{eqnarray*}
and therefore
$$W^{(1)}=\left[\tilde{V}D_f^{-1}D^{-2}D_{X\zeta i}\vdots \tilde{V}D^{-3}D_{Xfi} +  \check{V}D^{-2}D_{X\zeta i} \right]=$$
$$\left[\tilde{V}D_{X\zeta i}\vdots \tilde{V}D_{Xfi} +  \check{V}D D_{X\zeta i} \right]
\left[\begin{array}{llll}D_f^{-1}D^{-2} & 0\\0 & D^{-3}\end{array}\right].$$
Let us consider the matrix equation in the unknown $B$
$$W B=W^{(1)}.$$
As the right block of $W$ and the left block of $W^{(1)}$ are both equal to $\tilde{V}$ times a diagonal matrix, $B$ must have the form
$$B=\left[\begin{array}{ll}
0 &  B_{12} \\
D_B & B_{22} \end{array}\right],\,\;\;D_B=D_f^{-2}D^{-2}D_{X\zeta i}$$
and
$$\left[\begin{array}{l}
  B_{12} \\
 B_{22} \end{array}\right]=
W^{-1}\left(\tilde{V}D^{-3}D_{Xfi} +  \check{V}D^{-2}D_{X\zeta i} \right)=$$ $$\left[\begin{array}{ll}
I &  0 \\
0 &D_f^{-1} \end{array}\right][V\vdots \tilde{V}]^{-1}\tilde{V}D^{-3}D_{Xfi}+\left[\begin{array}{ll}
I &  0 \\
0 &D_f^{-1} \end{array}\right][V\vdots \tilde{V}]^{-1}\check{V}D^{-2}D_{X\zeta i}=$$
$$\left[\begin{array}{l}
  0 \\
D_f^{-1} \end{array}\right]D^{-3}D_{Xfi}+\left[\begin{array}{ll}
D^{-3} &  0 \\
0 &D_f^{-1}D^{-2} \end{array}\right]\tilde{X}D^{-1}D_{X\zeta i}$$
because it turns out that
$$[V\vdots \tilde{V}]^{-1}\tilde{V}=\left[\begin{array}{l}
  0 \\
 I \end{array}\right] \mbox{     and     } X \check{V}=\tilde{X}D$$
 and the elements of $\tilde{X}$ are polynomials in $\zeta_1,\dots,\zeta_p$;
therefore
$$B_{12}=D^{-3}\tilde{X}_{1}D^{-1}D_{X\zeta i} $$
and
$$B_{22}=D_f^{-1}D^{-3}D_{Xfi}+D_f^{-1} D^{-2}\tilde{X}_{2} D^{-1}D_{X\zeta i} .$$
But then
$$W^{-1}\frac{\partial W}{\partial s_h}W^{-1}=B W^{-1}=\left[\begin{array}{ll}
0 &  B_{12} \\
D_B & B_{22} \end{array}\right]\left[\begin{array}{ll}
D^{-3} &  0 \\
0 &D_f^{-1} D^{-2} \end{array}\right]\left[\begin{array}{ll}
X_{11} &  X_{12} \\
X_{21} & X_{22} \end{array}\right]=$$
$$\left[\begin{array}{ll}
0 &  B_{12}D_f^{-1} D^{-2} \\
D_B D^{-3}& B_{22}D_f^{-1} D^{-2} \end{array}\right]\left[\begin{array}{ll}
X_{11} &  X_{12} \\
X_{21} & X_{22} \end{array}\right]=$$
$$\left[\begin{array}{ll}
B_{12}D_f^{-1} D^{-2}X_{21} &  B_{12}D_f^{-1} D^{-2}X_{22} \\
D_B D^{-3}X_{11}+B_{22}D_f^{-1} D^{-2}X_{21}& D_B D^{-3}X_{12}+B_{22}D_f^{-1} D^{-2}X_{22} \end{array}\right]=$$
$$\left[\begin{array}{ll}
U_{11} & U_{12}\\
U_{21} & U_{22}
\end{array}\right]$$
where
$$
U_{11} =D^{-3}\tilde{X}_{1}D^{-1}D_{X\zeta i} D_f^{-1} D^{-2}X_{21} $$
$$U_{12}=  D^{-3}\tilde{X}_{1}D^{-1}D_{X\zeta i} D_f^{-1} D^{-2}X_{22} $$
$$U_{21}=
D_B D^{-3}X_{11}+\left(D_f^{-1}D^{-3}D_{Xfi}+D_f^{-1} D^{-2}\tilde{X}_{2} D^{-1}D_{X\zeta i}\right)D_f^{-1} D^{-2}X_{21}$$
$$ U_{22}=D_B D^{-3}X_{12}+\left(D_f^{-1}D^{-3}D_{Xfi}+D_f^{-1} D^{-2}\tilde{X}_{2} D^{-1}D_{X\zeta i}\right)D_f^{-1} D^{-2}X_{22} $$
Remembering that $\zeta_1,\dots,\zeta_p\in (0,1)$, it follows that $C_{jih}$ and $H_{jih}$ are rational functions such that the numerators are  polynomials in $\zeta_1,\dots,\zeta_p$ and a lower bound for the denominators is
$\prod_{r=1}^p f_r\prod_{r\ne s}(\zeta_r-\zeta_s)^6$ because some further simplification of common factors such as $(\zeta_r-\zeta_s)$ in the numerator and denominator can occur. This fact follows easily for $C_{jih}$ while for  $H_{jih}$  we remember that $$H_{jih}=C_{jih}\zeta_j+2 c_{ji} g_{ji}+f_j G_{jih}$$ and we notice that $U_{21}$ and $U_{22}$ are left multiplied by $D_f^{-1}$, therefore $f_jG_{jih}$ is a rational function such that the numerator is a polynomial in $\zeta_1,\dots,\zeta_p$
 and a lower bound for the denominator is $\prod_{r=1}^p f_r\prod_{r\ne s}(\zeta_r-\zeta_s)^5$.  Moreover as   $c_{ji}, g_{ji}$ do not depend on $\uf$ the claim follows as well as the thesis.
\end{proof}

As a final remark we notice that when the densities of $\bff_j\bzet_j$ and $\bff_j$ are approximately Gaussian, also their joint density is approximately Gaussian, because the density of $\bzet_j|\bff_j$ is approximately Gaussian too.

\section{Simulation results}

In order to illustrate the possible advantages of the proposed kernel estimator, the following MonteCarlo simulation was performed. $N=10^5$ independent realizations of a noisy multiexponential signal of length $n=126$ with three components
$$\bd_k^{(r)}=\sum_{j=1}^3 \zeta_j^{k-1}+\bep_k^{(r)},\;\; \uzet=[0.8,0.9,0.95],\;\;\sigma=1.5\cdot 10^{-3},\;k=1,\dots,n,\;r=1,\dots,N$$
were considered.  For  $r=1,\dots,N$ the $p=n/2$  generalized eigenvalues were computed as well as their empirical condensed density that was taken as the reference distribution that we want to estimate starting from the first $R=250$ samples $\bd_k^{(r)},\;r=1,\dots,R$.  The noise standard deviation $\sigma$ was chosen large enough to make at least one of the three modes hardly detectable by visual inspection in the empirical condensed density based on $R=250$ samples and small enough to make the three modes visually detectable in the reference condensed density. The number of observations was chosen as a function of $\sigma$ by the rule
$$n=\mbox{argmin}_k\{k| \;|d_k|<\sigma\}$$ as a compromise between the opposite requirements of a large sample size and a small total noise.

In the top part of Fig.1 the reference distribution evaluated in $256$ bins of equal size in the interval $(0.75,1)$ was plotted (right) as well as the empirical condensed density based on the first $R=250$ samples (left).
The kernel estimator $\hat{H}_G(x,t^+)$ was evaluated in $256$ equispaced points in the interval $(0.75,1)$ where $t^+$ is the estimated optimal bandwidth; the software downloadable by \cite{soft} was used and the result is plotted in Fig.1 (bottom left).
The kernel estimator (\ref{Pest}) was evaluated in the same points and plotted in Fig.1 (bottom  right). The estimated bandwidths were $t_0=1.1\cdot10^{-1},\;t^+=1.12\cdot 10^{-2}.$ We stress that in this problem what matters are the modes of the density because they are estimates of the generalized eigenvalues. A smooth estimate with the correct number of modes even if slightly displaced w.r. to the true values is much better than  an estimate with many modes not related to the true ones. Therefore
 we can conclude that the proposed estimate is much closer in a suitable Sobolev norm to the reference distribution than that based on standard diffusion.  Moreover if we compute the relative maxima of the proposed estimate above e.g. a threshold $\tau=2$ we get the modes $[0.82,0.88,  0.95 ]$ which are reasonable estimates of the true values $\uzet=[0.8,0.9,0.95]$.

 To stress the proposed method, a second example was considered where the signal has more and closer components. Moreover $\sigma$ was chosen large enough to make one of the modes visually undetectable even in the reference density. The multiexponential signal of length $n=324$ with five components was considered:
$$\bd_k^{(r)}=\sum_{j=1}^5 f_j\zeta_j^{k-1}+\bep_k^{(r)}$$
$$\uzet=[0.88,0.9,0.91,0.92,0.94],\;\;\uf=[1,10,10,10,1]$$
$$\sigma=2\cdot10^{-9},\;k=1,\dots,n,\;r=1,\dots,N$$
As before, $R=250$ samples were used. All the distributions were now evaluated in $2^{13}$ points in the interval $(0.85,0.96)$ and plotted in Fig.2. The estimated bandwidths were $t_0=3.8\cdot10^{-3},\;t^+=1.75\cdot 10^{-4}.$ In the reference distribution  one mode is lost, while the relative maxima above the threshold $\tau=2$ of the proposed estimate are
$[0.880, 0.902, 0.907, 0.919, 0.940 ].$

\section{Conclusions}

  The mathematical structure of the density of the ratio of Gaussian variables given by a partial differential equation has been revealed and exploited to solve a classical ill posed problem. The quality of the solution is definitely better than the one provided by classical methods. Moreover it turns out that, given a sample of observations of moderate size, the quality of the solution can be  better than the one obtained by a very large sample. The results are apparently robust with respect to the Normality hypothesis. It is reasonable to expect  that similar benefits can be obtained by exploiting the mathematical structure for solving other problems where the ratio of random variables plays an important role.



\begin{figure}
\begin{center}
\hspace{1.cm}{\fbox{\epsfig{file=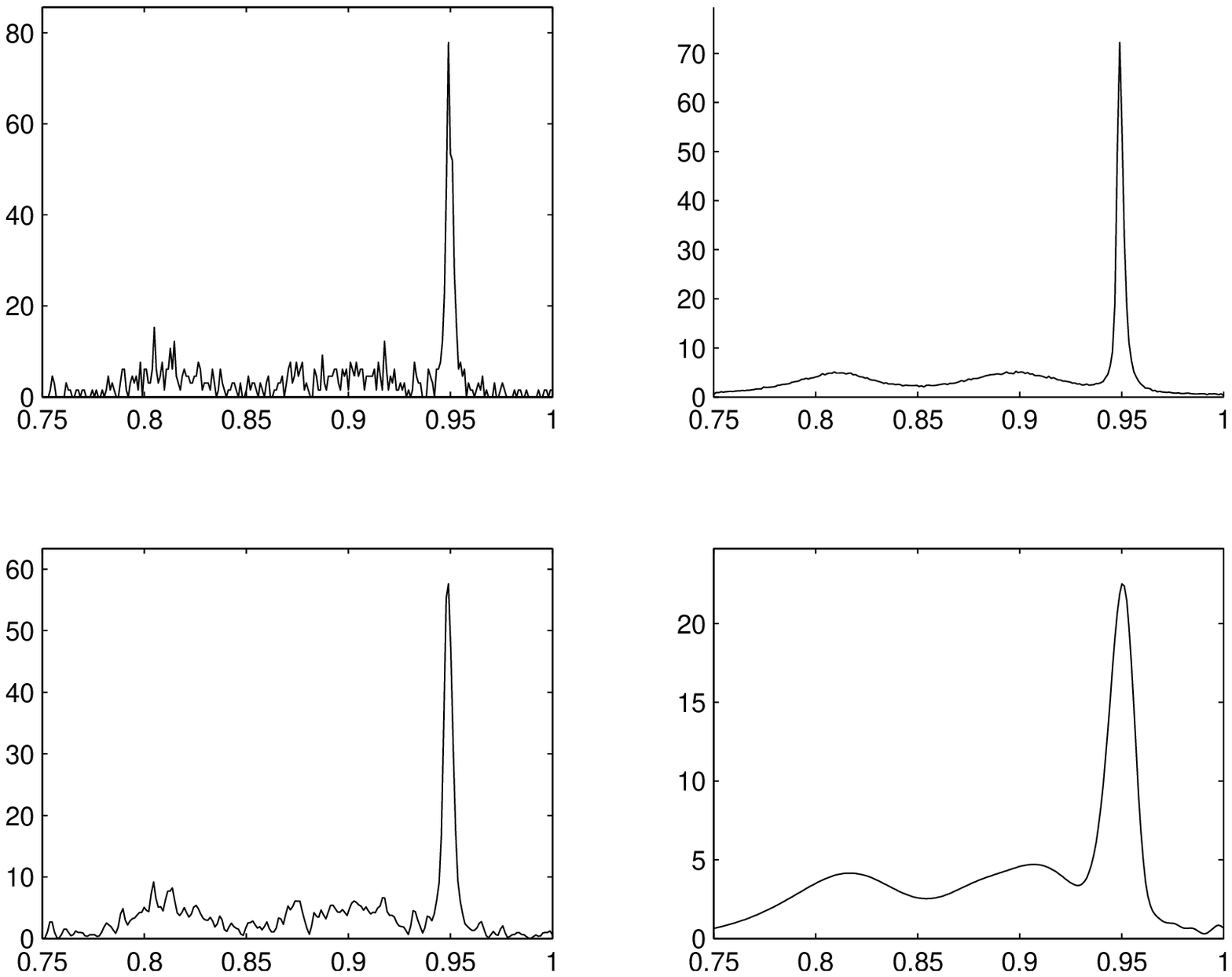,height=12cm,width=12cm}}}
\end{center}
\caption{Model 1. Top left: empirical distribution of the generalized eigenvalues based on $250$ data samples, evaluated on $256$ bins of equal size; top right: empirical distribution of the generalized eigenvalues based on $10^5$ data samples evaluated on the same bins; bottom left: Gaussian kernel density estimation based on $250$  samples, evaluated in $256$ equispaced points; bottom right: proposed kernel density estimation  based on $250$ samples, evaluated in the same points.
} \label{fig1}
\end{figure}

\begin{figure}
\begin{center}
\hspace{1.cm}{\fbox{\epsfig{file=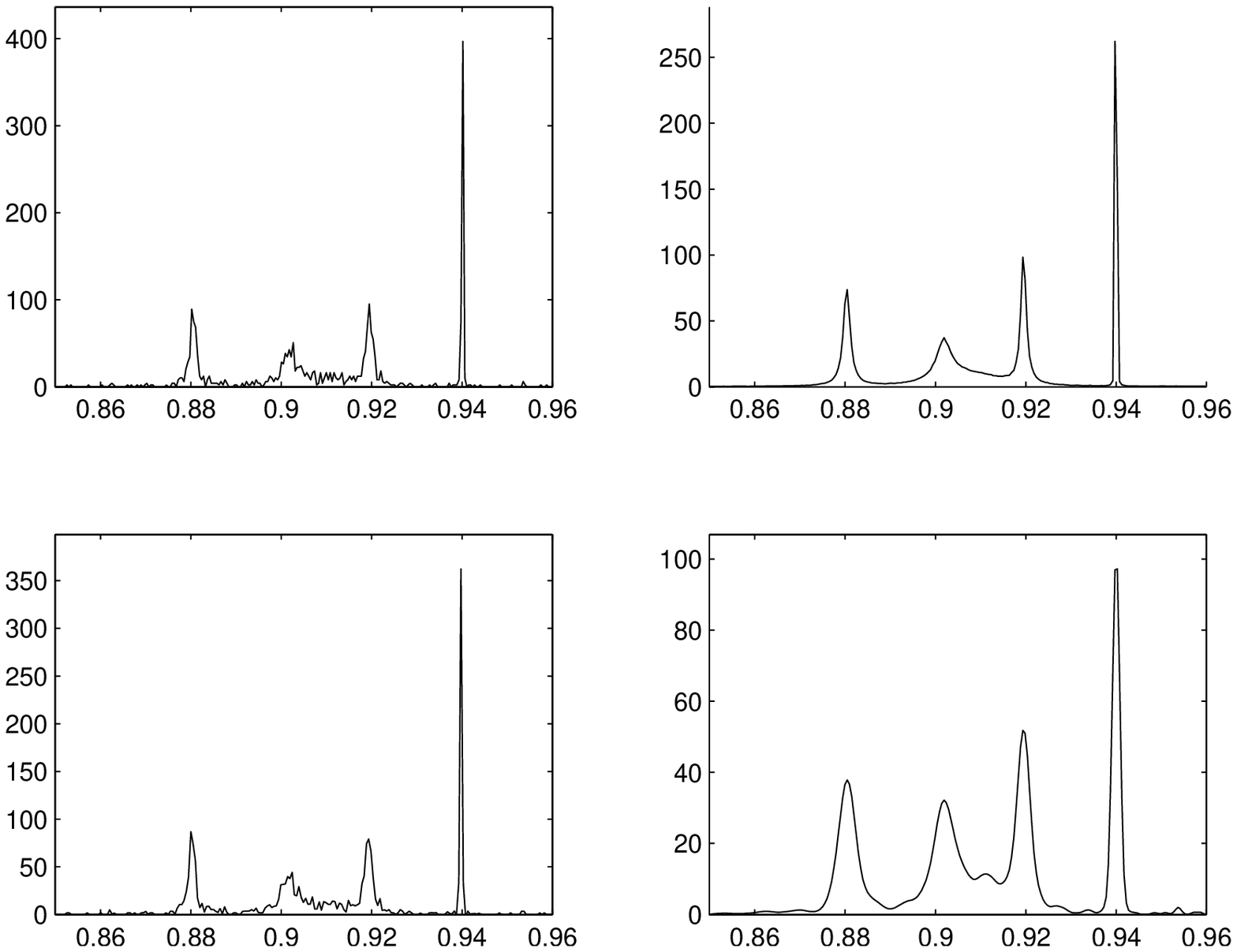,height=12cm,width=12cm}}}
\end{center}
\caption{Model 2. Top left: empirical distribution of the generalized eigenvalues based on $250$ data samples, evaluated on $2^{13}$ bins of equal size; top right: empirical distribution of the generalized eigenvalues based on $10^5$ data samples evaluated on the same bins; bottom left: Gaussian kernel density estimation based on $250$  samples, evaluated in $2^{13}$ equispaced points; bottom right: proposed kernel density estimation  based on $250$ samples, evaluated in the same points.
} \label{fig2}
\end{figure}

\end{document}